\documentclass{article}
\usepackage{graphicx} 
\usepackage{amsfonts,amsmath,amssymb,amsthm} 
\usepackage{enumerate}

\title{Optimal Codes and Arcs for the Generalized Hamming Weights}
\author{Sascha Kurz$^1$, Ivan Landjev$^2$, and Assia Rousseva$^3$}
\date{$^1$ Mathematisches Institut, Universit\"at Bayreuth, D-95440 Bayreuth, Germany\\$^2$
Institute of Mathematics and Informatics, Bulgarian Academy of Sciences, 8 Acad G. Bonchev str., 1113 Sofia, Bulgaria\\ New Bulgarian University, 21 Montevideo str, 1618 Sofia, Bulgaria\\$^3$ Sofia University, Faculty of Mathematics and Informatics, J. Bourchier Blvd., 1164 Sofia, Bulgaria\\}

\def\PG{\mathrm{PG}}
\def\F{\mathbb{F}}
\def\N{\mathbb{N}}
\def\wt{\operatorname{wt}}
\def\supp{\operatorname{supp}}
\def\cM{\mathcal{M}}
\def\cK{\mathcal{K}}

\newtheorem{corollary}{Corollary}
\newtheorem{theorem}{Theorem}
\newtheorem{definition}{Definition}
\newtheorem{proposition}{Proposition}
\newtheorem{lemma}{Lemma}
\newtheorem{example}{Example}
\theoremstyle{remark}

\begin{document}

\maketitle

\begin{abstract}
    \noindent
   This text contains some notes on the Griesmer bound. In particular, we 
   give a geometric proof of the Griesmer bound for the generalized weights
   and show that a Solomon--Stiffler type construction attains it if the minimum distance is sufficiently large. We also determine the parameters of optimal binary codes for dimensions at most seven and the optimal ternary codes for dimensions at most five.
  \medskip
  
  
  \smallskip
  
\end{abstract}

\section{Introduction}
\label{sec_introduction}

The Hamming weight of a codeword equals the size of its support and the minimum Hamming weight of a linear code is the minimum Hamming weight of the non-zero codewords. For a subcode the support us given by the set of all positions where at least one of the codewords in the subcode has a non-zero entry. With this, the $r$th generalized Hamming weight of a linear code is the size of the smallest support of an $r$-dimensional subcode. The generalized Hamming weights can be used to describe the cryptography performance of a linear code over the wire-tap channel of type~II \cite{wei1991generalized} and to determine the trellis complexity of the code \cite{chen2001trellis,forney1994density,forney1994dimension,kasami1993optimum}. From a geometrical point of view the $r$th generalized Hamming weight of a linear code corresponds to the number of points outside of a subspace of codimension $r$, where the points are the one dimensional subspaces spanned by the columns of a generator matrix of the linear code, see e.g.\ \cite{helleseth1992generalized,tsfasman1995geometric}. While one can easily find tables on the best known bounds for the parameters of linear codes with respect to the minimum Hamming distance for small parameters we were not able to find such a table for the $r$th generalized Hamming weight. The aim of this paper is to tabulate those numbers, where we will mainly use the geometric reformulation as multisets of points in projective spaces. In \cite{ball2025additive} the authors gave a construction for additive codes based on linear codes and their generalized Hamming weights. So, the constructions studied in this paper also give constructions for additive codes and the studied upper bounds show limitations for this specific construction. More precisely, for all sufficiently large distances a Griesmer type bound for additive codes can always be attained \cite{kurz2024additive}. As we will see, this also holds for linear codes and the $r$th generalized Hamming weight, but the mentioned construction of additive codes  results in optimal codes for a subset of the parameters only.

The remaining part of the paper is structured as follows. In Section~\ref{sec_preliminaries} we present the necessary preliminaries. In Section~\ref{sec_griesmer_bound} we analyze the Griesmer bound for the $r$th generalized Hamming distance. It turns out that a Griesmer code with respect to the Hamming distance also attains the Griesmer bound for the $r$th generalized Hamming distance. In  Section~\ref{sec_exactvalues}, we determine the minimum possible lengths of 
$[n,k]_q$ codes with given minimum $r$th generalized Hamming weight $d$ for 
some small parameters. We present our results in the geometric version, i.e.\ we determine the maximum number $m_q^{(r)}(k,w)$ of points in $\PG(k-1,q)$ such that each subspace of dimension $r$ contains at most $w$ points. We completely determine $m_2^{(r)}(k-1,w)$ for all $k\le 7$ and $m_3^{(r)}(k-1,w)$ for all $k\le 5$.

\section{Preliminaries}
\label{sec_preliminaries}

{\bf Linear codes.}\
A linear $[n,k]_q$ code $C$ is a $k$-dimensional subspace of $\F_q^n$. For $c=\left(c_1,\dots,c_n\right)\in\F_q^n$ we call
\begin{equation}
  \supp(c):=\left\{1\le i\le n\,:\, c_i\neq 0\right\}
\end{equation}
the \emph{support} of $c$ and $\wt(c)=|\supp(c)|$ its \emph{weight}. More generally,
for a vector subspace $C$ in $\F_q^n$, we define its \emph{support} as the set
of all coordinate positions in which the vectors of $C$ are not identically zero.
In other words,
\begin{equation}
  \supp(C):=\left\{1\le i\le n\,:\, \exists c=\left(c_1,\dots,c_n\right)\in C, c_i\neq 0\right\}.
\end{equation}
For two
$\F_q$-vector spaces $C$, $C'$ of $\F_q^n$ we write $C'\le C$ if $C'$ is a subspace of $C$.
The \emph{$r$th generalized Hamming weight} of
a linear code $C$ \cite{helleseth1977weight,klove1978weight}, denoted as $d_r(C)$, is the size of the
smallest support of an $r$-dimensional subcode of $C$ , i.e.\
\begin{equation}
   d_r(C):= \min\!\left\{|\supp(C')|\,:\, C'\le C, \dim(C')=r\right\}.
\end{equation}
In particular, $d_1(C)$ is the minimum Hamming distance of the $C$. 
A linear code of length $n$, dimension $k$ and $r$th generalized Hamming weight 
equal to $d_r$ will be called
an $[n,k,d_r]^{(r)}$-code.
The sequence $\left(d_1(C),\dots,d_k(C)\right)$ is called the \emph{weight hierarchy} of a linear $[n,k]_q$ code $C$. Clearly, $1\le d_1(C)< d_2< \dots < d_k(C)\le n$ 
(cf. e.g. \cite{wei1991generalized}). 


\medskip

\noindent
{\bf Multisets of points in $\mathbf{\PG(k-1,q)}$.}
A \emph{multiset} in $\PG(k-1,q)$ is a mapping
$\mathcal{K}\colon\mathcal{P}\to\mathbb{N}_0$, from the pointset $\mathcal{P}$ 
of $\PG(r,q)$ to the set of non-negative integers,
which assigns a multiplicity to each point of $\mathcal{P}$. This mapping is extended
to any subset $\mathcal{Q}$ of $\mathcal{P}$ by 
$\mathcal{K}(\mathcal{Q})=\sum_{P\in\mathcal{Q}} \mathcal{K}(P)$.
The integer $\mathcal{K}(\mathcal{Q})$ is called the multiplicity of 
$\mathcal{Q}$.

A multiset $\mathcal{K}$ is
called an \emph{$(n,w)$-arc} (resp.~\emph{$(n,u)$-minihyper}), if
$\mathcal{K}(\mathcal{P}) = n$,
$\mathcal{K}(H)\le w$ (resp. $\mathcal{K}(H)\ge u$)
for each hyperplane $H$ in $\PG(k-1,q)$, and
there is a hyperplane $H_0$ with $\mathcal{K}(H_0)=w$ (resp. $\mathcal{K}(H_0)=u$).

Let $\mathcal{K}$ be a multiset in $\PG(k-1,q)$. We denote by $w_r$ the maximal multiplicity of a $r$-dimensional subspace of $\PG(r,q)$ with respect to $\mathcal{K}$. In other words,
\begin{equation}
w_r:= \max \left\{\mathcal{K}(S)\,: S \text{ is a subspace of } \PG(k-1,q), \dim(S)=r \right\}.
\end{equation}
Similarly, we denote by $u_r$ the minimal multiplicity of a $r$-dimensional subspace of 
$\PG(r,q)$ with respect to $\mathcal{K}$. In other words,
\begin{equation}
u_r:= \min \left\{\mathcal{K}(S)\,: S \text{ is a subspace of } \PG(k-1,q), \dim(S)=r \right\}.
\end{equation}
A multiset in $\PG(k-1,q)$ of multiplicity $n$ is called an
$(n,w_r)^{(r)}$-arc (resp $(n,u_r)^{(r)}$-minihyper)
 if each $r$-dimensional subspace of $\PG(k-1,q)$ has multiplicity 
at most $w_r$ (resp. at least $u_r$), 
and there exists an $r$-dimensional subspace with this multiplicity.

It is known that there exists  a one-to-one correspondence between the 
isomorphism classes of the linear $[n,k]_q$-codes of full length 
(no coordinate is identically zero in all codewords) and the classes of 
projectively equivalent multisets in $\PG(k-1,q)$, where $k\ge2$.
The correspondence can be described as follows.
Let $C$ be a linear code of full length 
with parameters $[n,k,d]_q$, and let $G=({g}_1^T\cdots{g}_n^T)$,
${g}_i\in\mathbb{F}_q^k$, be a generator matrix of $C$.
The columns ${g}_i^T$ are considered as the homogeneous coordinates of points in
$\PG(k-1,q)$. In this way, the generator matrix $G$ is associated with an ordered
$n$-tuple of points $(P_1,\ldots,P_n)$ in $\PG(k-1,q)$. 
This $n$-tuple defines a multiset $\mathcal{K}$ by $\mathcal{K}(P):=|\{i\vert P_i=P\}|$. 
Clearly,
\begin{equation}
\label{eq:w+d}
w_r+d_{k-r-1}=n.
\end{equation}
This implies $w_0< w_1< \dots < w_{k-1}=n$. If $s\ge w_0$ is an integer, then
$\mathcal{K}'=s-\mathcal{K}$ is a multiset of cardinality $sv_{k}-n$ with
$u_r=sv_{r+1}-w_r$, $r=0,1,\dots,k-1$. Here $v_r=(q^r-1)/(q-1)$, as usual.

\begin{definition}
We denote by $m_q^{(r)}(k,w)$ the maximal multiplicity of a  multiset in $\PG(k-1,q)$
such that every $r$-dimensional subspace 
has multiplicity at most $w$. In other words,
$m_q^{(r)}(k,w)$ is defined as the largest $n$ for which there exists an
$(n,w)^{(r)}$-arc in $\PG(k-1,q)$.
\end{definition}

Double-counting directly gives:
\begin{lemma}
  \label{lemma_one_weight_bound_gen_hamming_weight}
  $m_q^{(r)}(k,w)\le \frac{v_k}{v_{k-r}}\cdot w$
\end{lemma}
Taking the union of two multisets of points gives:
\begin{lemma}
 \label{lemma_sum}
  $m_q^{(r)}(k,w_r'+w_r'')\ge m_q^{(r)}(k,w_r')+m_q^{(r)}(k,w_r'')$
\end{lemma}



In what follows we shall need multisets induced by a projection from a subspace.
Let $\mathcal{K}$ be a multiset in $\PG(k-1,q)$.
Fix an $i$-dimensional subspace $\delta$ in
$\PG(k-1,q)$, with $\mathcal{K}({\delta})=t$.
Let further $\pi$ be a $j$-dimensional subspace in $\PG(k-1,q)$ of complementary dimension,
i.e. $i+j=k-2$ and $\delta\cap\pi=\varnothing$.
Define the projection $\varphi=\varphi_{\delta,\pi}$ from $\delta$ onto $\pi$ by
\begin{equation}
\label{eq:project}
\varphi\colon
\left\{\begin{array}{lll}
\mathcal{P}\setminus\delta\ & \rightarrow\ & \pi \\
Q & \rightarrow & \pi\cap\langle\delta,Q\rangle .
\end{array}\right.
\end{equation}
As before, $\mathcal{P}$ denotes the set of all points of $\PG(k-1,q)$.
Note that $\varphi$ maps $(i+s)$-subspaces containing
$\delta$ into $(s-1)$-subspaces in $\pi$.
Denote by $\mathcal{P}'$ the set of all points in $\pi$.
We define the induced multiset
$\mathcal{K}^{\varphi}:\mathcal{P}'\to\mathbb{N}_0$ by
\[\mathcal{K}^{\varphi}(Q)\,=
\,\sum_{P:\varphi(P)=Q} \mathcal{K}(P).\]
It is clear that
For every subspace $S$ in $\PG(k-1,q)$
that contains $\delta$, it holds
$\mathcal{K}^{\varphi}(\varphi(S))=\mathcal{K}(S)-t$.
In particular, if $\mathcal{K}$ is a $(n,w_r)^{(r)}$-arc (resp. $(n,u_r)^{(r)}$-minihyper)
then $\mathcal{K}^{\varphi}$ is an 
$(n-t,w_r-t)^{(r)}$-arc (resp. an $(n-t,u_r-t)^{(r)}$-minihyper).

\section{The Griesmer bound}
\label{sec_griesmer_bound}
It is well known that the Griesmer bound \cite{griesmer1960}:
\begin{equation}
  \label{eq_griesmer_bound}
  n\ge \sum_{i=0}^{k-1} \left\lceil\frac{d}{q^i}\right\rceil=:g_q(k,d).
\end{equation}
for the minimum length of an $[n,k]_q$ code with given minimum Hamming distance $d$ is attained if $d$ is sufficiently large. A similar estimate holds for the
$r$th generalized Hamming distance. This result was proved in
\cite{helleseth1992generalized,helleseth1995etal}. The corresponding bound is called {\em the generalized Griesmer 
bound}. Similarly to the classical Griesmer bound, this result is of purely geometric nature.
%

Below, we give a geometric proof of the generalized Griesmer bound.

\begin{theorem}
\label{thm_griesmer_gen_ham}
Let $C$ be an $[n,k,d_r]^{(r)}$-code, where $1\le r\le k$. Then
\begin{equation}
\label{eq:ggriesmer}
n\ge d_r+\sum_{i=1}^{k-r}\left\lceil\frac{d_r}{q^iv_r}\right\rceil.
\end{equation}
\end{theorem}

\begin{proof}
The result is clearly true for $r=k$.
It is enough to prove it for codes of full length, i.e. $n=d_k$.

Let $C$ be an $[n,k,d_r]_q^{(r)}$ code and let $\cK$ be the $(n,w_{k-r-1})^{(k-r-1)}$-arc in $\PG(k-1,q)$ associated with it.
Consider a point of maximum multiplicity $w_0$. The total multiplicity of the remaining points is $n-w_0$, and there must be point among them of multiplicity at least
$\displaystyle \frac{n-w_0}{v_k-1}$. Hence
\[w_0\ge \left\lceil\frac{n-w_0}{qv_{k-1}}\right\rceil.\]
This is the Griesmer inequality for $r=k-1$ since $w_0=n-d_{k-1}$.

Without loss of generality, we shall assume that $C$ is not extendable in the following sense.
Every point in $\PG(k-1,q)$ is contained 
in a maximal subspace of dimension $r$.

Fix a $(k-r-1)$-dimensional subspace $S$ of maximal multiplicity, i.e. multiplicity $w_{k-r-1}$.
There exist $v_{k+1}-v_{k-r}$ points in $\PG(k-1,q)$ outside $S$. Hence there is a point 
$P\notin S$ of multiplicity 
\[t\ge\frac{n-w_{k-r-1}}{q^{k-r}v_r}=\frac{d_r}{q^{k-r}v_r}.\]
Consider a projection $\varphi$ from $P$ onto some hyperplane $H\cong\PG(k-2,q)$ with
$P\notin H$. The induced arc $\cK^{\varphi}$ has parameters
$(n-t,w_r-t)^{(r)}$.
The image of every subspace $T$ through $P$ is a subspace in $H$ of dimension $\dim T-1$. Moreover, 
\[\mathcal{K}^{(\varphi)}(\varphi(T))=\mathcal{K}(T)-t.\]
If we set $w_i'=\max_{T:\dim T=i} \cK^{\varphi}(T)$, we have 
$w_i'\le w_{i+1}-t$. If we also set
$d_i'=n-t-w'_{k-i-2}$, we get $d_r\le d_r'$.
By the induction hypothesis
\[n'\ge d'_{r-1}+\sum_{i=1}^{k-r-1} \lceil\frac{d'_{r-1}}{q^iv_{r-1 }}\rceil.
\]
Now using $d_r\le d'_r$, we obtain:
\begin{eqnarray*}
n'=n-t&=& d_r'+\sum_{i=1}^{k-r-1} \left\lceil\frac{d_i'}{q^iv_r}\right\rceil \\
      &=& d_r+\sum_{i=1}^{k-r-1} \left\lceil\frac{d_i}{q^iv_r}\right\rceil \\
\end{eqnarray*}
This implies
\begin{eqnarray*}
n&\ge & d_r+ t + \sum_{i=1}^{k-r-1} \left\lceil\frac{d_i}{q^iv_r}\right\rceil \\
 &\ge & d_r+\left\lceil\frac{d_r}{q^{k-r}v_r}\right\rceil+ \sum_{i=1}^{k-r-1} \left\lceil\frac{d_i}{q^iv_r}\right\rceil \\
 &=& d_r+ \sum_{i=1}^{k-r} \left\lceil\frac{d_i}{q^iv_r}\right\rceil.
\end{eqnarray*}
\end{proof}

We shall denote the left-hand side of (\ref{eq:ggriesmer}) by $g_r^{(r)}(k,d_r)$. The following theorem is well-known (cf.\ e.g.~\cite{rousseva2019linear})

\begin{theorem}
\label{thm:genweights}
Let $\cK$ be an $(n,n-d)$-arc in $\PG(k-1,q)$ and let $C$ be an
$[n,k,d]_q$ code associated with $\cK$, If $n=t+g_q(k,d)$, then
\begin{equation}
w_j=t+\sum_{i=k-1-j}^{k-1} \left\lceil\frac{d}{q^i} \right\rceil
\end{equation}
\end{theorem}

Let $k\ge1$ and let $d$ be a positive integer.
The integer $d$ can be written uniquely as
\begin{equation}
\label{eq:d}
d=\sigma q^{k-1}-\sum_{i=0}^{k-2} \varepsilon_iq^i
\end{equation}
It is easily computed that
\begin{equation}
\label{eq:n}
g_q^{(1)}(k,q)=\sigma v_k-\sum_{i=0}^{k-2} \varepsilon_iv_{i+1}.
\end{equation}
Let $C$ be an $[n=g_q(k,d),k,d]_q^{(1)}$-code and let $\cK$ an arc
$(n,n-d)$-arc in $\PG(k-1,q)$ associated with $C$. By Theorem~\ref{thm:genweights},

\begin{equation}
\label{eq:w}
w_{k-1-j}=\sigma v_{k-j}-\sum_{i=j}^{k-2} \varepsilon_iv_{i+1-j}.
\end{equation}

Since the maximal multiplicity of a point with respect to $\cK$ is $\sigma$,
$\mathcal{K}$ can be represented as
$\mathcal{K}=\sigma\chi_{\mathcal{P}}-\mathcal{F}$, where $\mathcal{F}$ is a minihyper with
parameters
\begin{equation}
\label{eq:minihyp}
\left(\sum_{i=0}^{k-2} \varepsilon_iv_{i+1},\sum_{i=0}^{k-2} \varepsilon_iv_{i}
\right)
\end{equation} 
with maximal multiplicity of a point not exceeding $\sigma$.
Minihypers with the above parameters, but without the restriction on the maximal point multiplicity,
can always  can be constructed as the sum of subspaces: 
$\varepsilon_{k-2}$ hyperplanes, $\varepsilon_{k-3}$ hyperlines, and so on, $\varepsilon_1$ lines and $\varepsilon_0$ points.
Such minihypers are called \emph{canonical}.

A canonical minihyper $\mathcal{F}$ in $\PG(k-1,q)$ is said to be of type
$\sum_{i=1}^{k-2} \varepsilon_{i}[i-1]$ if it is the sum $\varepsilon_{k-2}$ hyperplanes,
$\varepsilon_{k-3}$-hyperlines, and so on $\varepsilon_2$ lines, and  $\varepsilon_1$ points.
In other words, $\mathcal{F}=\sum_i\chi_{S_i}$, where exactly $\varepsilon_j$ of the subspaces $S_i$ are of (projective) dimension $j-1$. The arc $\cK$ obtained by subtracting $\mathcal{F}$ from $s$ copies of $\PG(K-1,q)$ will be said to have type
$s[k-1]-\sum_{i=1}^{k-2}[\varepsilon_{i}]$.

Minihypers with parameters
\begin{equation}
\label{eq:minihyp2}
\left(\sum_{i=r}^{k-2} \varepsilon_iv_{i+1},\sum_{i=r}^{k-2} \varepsilon_iv_{i+1-r}
\right)^{(k-1-r)}
\end{equation}
can be constructed in a similar fashion as canonical minihypers, i.e. as the sum of 
$\varepsilon_{k-2}$ hyperplanes, $\varepsilon_{k-3}$ hyperlines, and so on, $\varepsilon_r$ subspaces of dimension $r$.

\begin{theorem}
\label{thm:minihyp-griesmer}
Let $\mathcal{F}$ be a minihyper in $\PG(k-1,q)$ with 
parameters given by (\ref{eq:minihyp2}), $0\le\varepsilon_\le q-1$, and with maximal point
multiplicity $w_0$. Then for every $\sigma\ge w_0$ the multise
$\mathcal{K}=\sigma-\mathcal{F}$ is a (multi)arc in $\PG(k-1,q)$ with parameters
\[\left(n=\sigma v_k-\sum_{i=r}^{k-2} \varepsilon_iv_{i+1},
w_{k-r-1}=\sigma v_{k-r}-\sum_{i=r}^{k-2} \varepsilon_iv_{i+1-r} \right)^{(k-1-r)}.
\]
The code associated with $\mathcal{K}$ is a Griesmer code with respect to the
$r$th generalized Hamming weight.
\end{theorem}

\begin{proof}
The parameters of $\mathcal{K}$
are obvious. Furthermore 
\[d_r=n-w_{k-r-1}=v_r(\sigma q^{k-r}-\sum_{i=r}^{k-2}\varepsilon_iq^{i-r+1}).\]
Now it is a straightforward check that:
\begin{eqnarray*}
g_q^{(r)}(k,r) &=& d_r+\lceil\frac{d_r}{qv_r}\rceil+\cdots+\lceil\frac{d_r}{q^{k-r}v_r}\rceil\\
 &=& v_r(\sigma q^{k-r}-\sum_{i=r}^{k-2}\varepsilon_iq^{i-r+1}) + \\
  && \sigma q^{k-r-1}-\sum_{i=r}^{k-2}\varepsilon_i^{i-r}+ \\
  && \sigma q^{k-r-2}-\sum_{i=r+1}^{k-2}\varepsilon_i^{i-r-1} +\ldots +\sigma \\
  &=& \sigma v_k-\sum_{i=r}^{k-2} \varepsilon_iv_{i+1} \\
  &=& n
\end{eqnarray*}
\end{proof}

If the maximal point multiplicity of $\mathcal{F}$ with parameters given by
(\ref{eq:minihyp2}) is large, say at least
$1+\sum_{i=r}^{k-2} i\varepsilon_i$,
such minihypers can always be constructed. Thus we have the following theorem, which is well-known, and essentially due to Solomon and Stiffler \cite{solomonstiffler1965}.

\begin{theorem}
\label{thm:larged}
If $d_r$ is large enough Griesmer 
$[n=g_q^{(r)}(k,d_r)]_q^{(r)}$-codes do exist for all $q, k$ and $r$.
\end{theorem}

Next we are going to demonstrate that if a code attains the classical Griesmer bound, it attains 
also the Griesmer bound for all generalized weights.

\begin{theorem}
\label{thm:grriesmer1}
Let $C$ be an $[n=g_q(k,d),k,d]_q^{(1)}$-code where $d^{(1)}=d$ is given by (\ref{eq:d}).
Then 
\[d_r(C)=v_r\cdot\left(\sigma q^{k-r}-\sum_{i=r}^{k-1} \varepsilon_{i-1}q^{i-r}\right)-
\sum_{i=1}^{r-1} \varepsilon_{i-1}v_i,\]
and $C$ attains the Griesmer bound for the $r$th generalized Hamming weight, i.e.
$n=g_q^{(r)}(k,d_r)$.
\end{theorem}

\begin{proof}
By (\ref{eq:w}) and (\ref{eq:w+d}), and using the obvious
$v_{i+r}-v_i=q^iv_r$, we get that
\begin{eqnarray*}
d_r &=& n-w_{k-1-r} \\
    &=& \left(\sigma v_k-\sum_{i=1}^{k-1} \varepsilon_{i-1}v_i\right)-
\left(\sigma v_{k-r}-\sum_{i=r}^{k-1} \varepsilon_{i-1}v_{i-r}\right) \\
    &=& \sigma(v_k-v_{k-r})-\sum_{i=r}^{k-1} \varepsilon_{i-1}(v_i-v_{i-r})-
        \sum_{i=1}^{r-1} \varepsilon_{i-1}v_i \\
    &=& v_r\cdot\left(\sigma q^{k-r}-\sum_{i=r}^{k-1} \varepsilon_{i-1}q^{i-r}\right)-
\sum_{i=1}^{r-1} \varepsilon_{i-1}v_i.
\end{eqnarray*}
Let us note that
\[\sum_{i=1}^r \varepsilon_{i-1}v_i\le(q-1)\sum_{i=1}^r v_i<(q-1)v_r+v_r=qv_r.\]
Now it is a straightforward check that
\[g_q^{(r)}(k,d_r)=g^{(1)}(k,d_1).\]
\end{proof}

\section{Exact values}
\label{sec_exactvalues}

In this section, we tackle the problem finding the exact value of 
$m_q^{(r)}(k-1,w)$ for fixed $k, r,  w$ and $q$. 
It is in general hard to determine the values $m_q^{(r)}(k-1,w)$, but 
there is one easy case for which the result is obvious:

\begin{proposition}
  We have $m_q^{(0)}(k-1,s)=sv_k$ for each $k\ge 1$ and each $s\ge 1$.
\end{proposition}

The case $r=1$ is of special interest. This is the problem of determinig the largest size of a generalized cap -- a (multi)set of points such that each line has multiplicity at most $w$.
In particular if $w=2$ this is the notorious maximal cap problem. 

From the connection between the linear $[n,k]_q$ codes and the multisets of points in
$\PG(k-1,q)$, we get that $m^{(r)}_q(k-1,w)$ is the largest integer $n$ such that 
$n\ge g_q^{(k-r-1)}(k,n-w)$. This integer will be called 
the \emph{Griesmer upper bound} for $m^{(r)}_q(k-1,w)$.

Another bound on $m_q^{(r)}(k-1)$ is the following.
Set  $s_r=w$ and $s_{r+i}=m_q^{(r+i-1)}(r+i,s_r)$, $i=0,\ldots,k-1-r$. Then 
$m_q^{(r)}(k-1)\le s_{k-1}$. We call this bound
the \emph{coding upper bound} for $m^{(r)}_q(k-1,w)$.
In other words, the coding upper bound uses recursively
the parameters of optimal multiarcs, which in turn can be obtained from the parameters of the
optimal linear codes.

\begin{example}
  For $m_2^{(4)}(6,21)$ the Griesmer upper bound is $81$ since
 \[81\ge g_2^{(2)}(7,60)=60+\left\lceil\frac{60}{6}\right\rceil+
 \left\lceil\frac{60}{12}\right\rceil+\left\lceil\frac{60}{24}\right\rceil+\left\lceil\frac{60}{48}\right\rceil+
 \left\lceil\frac{60}{96}\right\rceil=81,\]
 while
  \[82< g_2^{(2)}(7,61)=61+\left\lceil\frac{61}{6}\right\rceil+
 \left\lceil\frac{61}{12}\right\rceil+\left\lceil\frac{61}{24}\right\rceil+\left\lceil\frac{61}{48}\right\rceil+
 \left\lceil\frac{61}{96}\right\rceil=84.\]

  The coding upper bound is $77$. Here $s_4=21$, $s_5=m_2^{(4)}(5,21)=39$ since there exists a binary $[39,6,18]$-code and there is no $[40,6,19]$-code (cf. Grassl's tables 
  \cite{grassl-tables}).
Furthermore, $s_6=m_2^{(5)}(6,39)=77$ since there exists a $[77,6,38]_2$-code and there is no
$[78,6,39]_2$-code.
\end{example}


In Subsection~\ref{subsec_binary} we determine the exact values of $m^{(r)}_2(k-1,w)$ for all $k\le 7$. We remark that $m^{(6)}_2(7,w)$ is completely known 
(by the fact that we know the optimal lengths of all binary 8-dimensional codes)
while there only partial results for $m^{(7)}_2(8,w)$. However, the determination of $m^{(6)}_2(7,w)$ is scattered in many papers, so that we do not attempt to determine $m^{(r)}_2(8,w)$ here. In Subsection~\ref{subsec_ternary} we determine the exact values of 
$m^{(r)}_3(k-1,w)$ for all $k\le 5$.

\subsection{Exact values for $m_2^{(r)}(k-1,w)$}
\label{subsec_binary}

\begin{proposition}
  \label{prop_n_2_4_2_s}
  It holds:
  \begin{enumerate}[(a)]
  \item $m_2^{(1)}(3,3t+2)=15t+8$, 
  \item $m_2^{(1)}(3,3t+3)=15t+15$, 
  \item $m_2^{(1)}(3,3t+4)=15t+16$,
\end{enumerate}
for all $t\in N$.
\end{proposition}
\begin{proof}
  The upper bounds are given by the Griesmer upper bound. Constructions for 
  $m_2^{(1)}(3,2)\ge 8$,
  $m_2^{(1)}(3,3)\ge 15$, and $m_2^{(1)}(3,4)\ge 16$ are given by an affine solid, 
  a solid, and a solid plus a point, respectively. Combining those examples with $t$ copies of a solid yields the remaining upper bounds by Lemma~\ref{lemma_sum}.
\end{proof}

\begin{table}[htp]
  \begin{center}
    \begin{tabular}{rrcc}
      \hline 
      $w$ & $m_2^{(1)}(3,w)$ & construction & upper bound \\ 
      \hline
      2 &  8 & affine solid & Griesmer upper bound \\ 
      3 & 15 & solid & Griesmer upper bound \\ 
      4 & 16 & solid plus a point & Griesmer upper bound \\ 
      \hline
    \end{tabular}  
    \caption{Exact values for $m_2^{(1)}(3,w)$.}
    \label{table_n_2_4_2_s}
  \end{center}
\end{table}

\begin{corollary}
  $m_2^{(1)}(3,w)$ is given by the Griesmer upper bound for all $w\ge 2$.
\end{corollary}

Proposition~\ref{prop_n_2_4_2_s} can be generalized:
\begin{proposition}
  For each $r\ge 2$ and each $t\in \N$ we have  
\begin{enumerate}[(a)]  
\item $m_2^{(1)}(k-1,3t+2)=tv_k+2^{k-1}$, 
\item $m_2^{(1)}(k-1,3t+3)=(t+1)v_k$, 
\item $m_2^{(1)}(k-1,3t+4)=(t+1)v_k+1$.
\end{enumerate}    
\end{proposition}
\begin{proof}
  Constructions for $m_2^{(1)}(k-1,2)\ge 2^{k-1}$, $m_2^{(1)}(k-1,3)\ge v_k$,
  and $m_2^{(1)}(k-1,2)\ge v_k+1$ are given by an affine $(k-1)$-space  (type, a $(k-1)$-space,  and an $(k-1)$-space plus a point, respectively. Combining those examples with $t$ copies of a $(k-1)$-space yields the remaining upper bounds by Lemma~\ref{lemma_sum}.

  The upper bounds are given by the Griesmer upper bound. 
More precisely, applying Theorem~\ref{thm_griesmer_gen_ham} with 
$d_{k-2}=tv_k+2^{k-1}-3t-2=(2t+1)\cdot 2v_{k-2}$ gives 
$n\ge (2t+1)\cdot 2v_{k-2}+(2t+1)+(t+1)=tv_k+2^{k-1}$ and applying Theorem~\ref{thm_griesmer_gen_ham} with $d_{k-2}=(2t+1)\cdot2v_{k-2}+1$ gives 
$n\ge (2t+1)\cdot 2[r]_2+(2t+2)+(t+1)=tv_k+2^{k-1}+2$.
The other two cases are treated in a similar fashion.
\end{proof}

\begin{corollary}
  For each $k\ge 3$ we have that $m_2^{(1)}(k-1,w)$ is given by the Griesmer upper bound for all $w\ge 2$.
\end{corollary}

\begin{lemma}
  \label{lemma_projective_base}
  For $k\ge 5$ we have $m_2^{(k-3)}(k-1,k-2)=k+1$. 
\end{lemma}
\begin{proof}
  A projective base (or frame) shows $m_2^{(k-3)}(k-1,k-2)\ge k+1$. From the Griesmer upper bound we conclude $m_2^{(k-3)}(k-2,k-2)\le k$, so that we can assume that a multiset $\cM$ of at least $k+1$ points contains the points spanned by the $k$ unit vectors. Any further point in $\cM$ spanned by $v\in\F_2^k$ then needs to have Hamming weight $k-1$ or $k$, since otherwise $k-1$ points would be contained in a subspace of codimension two. The sum of two different such vectors with Hamming weight $k$ or $k-1$ has Hamming weight strictly less than $k-1$, so that we can find $k-1$ points in a subspace of codimension two. 
\end{proof}
Note that ovoids imply $m_q^{(1)}(3,2)\ge q^2+1$.

\begin{proposition}
  \label{prop_n_2_5_1_2}
  We have 
\begin{enumerate}[(a)]  
\item  $m_2^{(2)}(4,7t+4)=31t+16$, 
\item $m_2^{(2)}(4,7t+5)=31t+17$, 
\item $m_2^{(2)}(4,7t+6)=31t+24$, 
\item $m_2^{(2)}(4,7t+7)=31t+31$, 
\item $m_2^{(2)}(4,7t+8)=31t+32$, 
\item $m_2^{(2)}(4,7t+9)=31t+33$, 
\item $m_2^{(2)}(4,7t+10)=31t+40$ 
 \end{enumerate} 
 for all $t\in N$. Moreover, we have $m_2^{(2)}(4;3)=6$.
\end{proposition}
\begin{proof}
 Lemma~\ref{lemma_projective_base} yields $m_2^{(2)}(4,3)=6$. Constructions for $4\le s\le 10$ are given by multisets of points with types $[4]-[3]$, $[4]-[3]+[0]$, $[4]-[2]$, $[4]$, 
 $[4]+[0]$, $[4]+2[0]$, and $2[4]-[3]-[2]$, respectively. Combining those examples with $t$ copies of a $4$-space yields the remaining constructions by Lemma~\ref{lemma_sum}.
  The upper bounds for $m_2^{(2)}(4,w)$ are given by the Griesmer upper bound for all $w\ge 4$. 
\end{proof}

\begin{table}[htp]
  \begin{center}
    \begin{tabular}{rrcc}
      \hline 
      $w$ & $m_2^{(2)}(4,w)$ & construction & upper bound \\ 
      \hline
      3 &  6 & projective base & Lemma~\ref{lemma_projective_base} \\ 
      4 & 16 & $[4]-[3]$ & Griesmer upper bound \\ 
      5 & 17 & plus point & Griesmer upper bound \\
      6 & 24 & $[4]-[2]$ & Griesmer upper bound \\
      7 & 31 & $[4]$ & Griesmer upper bound \\ 
      8 & 32 & plus point & Griesmer upper bound \\ 
      9 & 33 & plus point & Griesmer upper bound \\ 
      10 & 40 & $2[4]-[3]-[2]$ & Griesmer upper bound \\
      \hline
    \end{tabular}  
    \caption{Exact values for $m_2^{(2)}(4,w)$.}
    \label{table_n_2_5_2_s}
  \end{center}
\end{table}

\begin{corollary}
  $m_2^{(2)}(4,w)$ is given by the Griesmer upper bound for all $w\ge 4$.
\end{corollary}

We can generalize Proposition~\ref{prop_n_2_5_1_2} as follows:
\begin{proposition}
  For each $k\ge 4$ and all $w\ge 4$ we have that $m_2^{(2)}(k-1,w)$ is given by the Griesmer upper bound.    
\end{proposition}
\begin{proof}
  The Solomon--Stiffler constructions for $[k-1]-[k-2]$, $[k-1]-[k-3]$, $[k-1]$, and $2[k-1]-[k-2]-[k-3]$ give $m_2^{(k-4)}(k-1,4)\ge 2^{k-2}$, $m_2^{(k-4)}(k-1,6)\ge 3\cdot 2^{k-3}$, $m_2^{(2)}(k-1,7)\ge v_k$, and $m_2^{(2)}(k-1,10)\ge 5\cdot 2^{k-3}$. Adding points gives $m_2^{(2)}(k-1,5)\ge m_2^{(2)}(k-1,4)+1$, $m_2^{(2)}(k-1,8)\ge m_2^{(k-4)}(k-1,7)+1$, and  $m_2^{(2)}(k-1,9)\ge m_2^{(2)}(k-1,7)+2$. For $w>10$ the lower bounds 2 given by $m_2^{(2)}(k-1,w)\ge m_2^{(2)}(k-1,w-7)+m_2^{(2)}(k-1,7)$.

  The upper bounds are given by the Griesmer upper bound. More precisely, applying Theorem~\ref{thm_griesmer_gen_ham} with $d_{k-3}=(4t+2)\cdot 2v_{k-3}$ gives 
  $n\ge (4t+2)\cdot 2v_{k-3}+(4t+2)+(2t+1)+(t+1)=t\cdot v_+2^{k-1}$; 
  applying Theorem~\ref{thm_griesmer_gen_ham} with $d_r=(4t+3)\cdot 2v_{k-3}$ gives 
  $n\ge (4t+3)\cdot 2v_{k-3}+(4t+3)+(2t+2)+(t+1)=t\cdot v_k+3\cdot 2^{k-2}$; 
  applying Theorem~\ref{thm_griesmer_gen_ham} with $d_r=(4t+4)\cdot 2v_{k-3}$ gives 
  $n\ge (4t+4)\cdot 2v_{k-3}+(4t+4)+(2t+2)+t=t\cdot v_k+ v_k$, 
  and applying Theorem~\ref{thm_griesmer_gen_ham} with $d_r=(4t+5)\cdot 2v_{k-3}$ gives 
  $n\ge (4t+5)\cdot 2v_{k-3}+(4t+5)+(2t+3)+(t+2)=t\cdot v_k+5\cdot 2^{k-2}$.
\end{proof}
So, the only non-trivial value that is not determined yet is $m_2^{(2)}(k-1,3)$. The first values are given by  $m_2^{(2)}(3,3)=5$ and $m_2^{(2)}(4,3)=6$. 

\begin{table}[htp]
  \begin{center}
    \begin{tabular}{rrcc}
      \hline 
      $w$ & $m_2^{(2)}(5,w)$ & construction & upper bound \\ 
      \hline
      3 &  8 & Lemma~\ref{lemma_n_2_6_3_3}  & Lemma~\ref{lemma_n_2_6_3_3} \\ 
      4 & 32 & $[5]-[4]$ & Griesmer upper bound \\ 
      5 & 33 & plus point & Griesmer upper bound \\
      6 & 48 & $[5]-[3]$ & Griesmer upper bound \\
      7 & 63 & $[5]$ & Griesmer upper bound \\
      8 & 64 & plus point & Griesmer upper bound \\
      9 & 65 & plus point & Griesmer upper bound \\
      10 & 80 & $2[5]-[4]-[3]$ & Griesmer upper bound \\
      \hline
    \end{tabular}  
    \caption{Exact values for $m_2^{(2)}(5,w)$.}
    \label{table_n_2_6_3_s}
  \end{center}
\end{table}

\begin{lemma}
  \label{lemma_n_2_6_3_3}
  We have $n_2^{(2)}(5,3)=8$.
\end{lemma}
\begin{proof}
  A feasible example is given by 
  $$
    \begin{pmatrix}
      1 & 0 & 0 & 0 & 0 & 0 & 1 & 1 \\
      0 & 1 & 0 & 0 & 0 & 0 & 1 & 1 \\
      0 & 0 & 1 & 0 & 0 & 0 & 1 & 0 \\
      0 & 0 & 0 & 1 & 0 & 0 & 1 & 0 \\
      0 & 0 & 0 & 0 & 1 & 0 & 0 & 1 \\
      0 & 0 & 0 & 0 & 0 & 1 & 0 & 1 \\
    \end{pmatrix}
  $$
  Let $\cM$ be a multiset of $n\ge 9$ points in $\PG(5,2)$ such that each plane contains at most three points. Since $m_2^{(2)}(5,3)=6$ we assume w.l.o.g.\ that $\cM$ contains the six points spanned by the six unit vectors. Clearly, the maximum point multiplicity is one and every additional point is spanned by a vector $x$ with Hamming weight at least $4$. Since $m_2^{(3)}(5,2;4)=7$ we assume w.l.o.g.\ that $\cM$ also contains the point spanned by $x=(1,1,1,1,0,0)^\top$. Let two further points be spanned by $y,z\in\F_2^6$ . Since every plane contains at most three points we have $\wt(y),\wt(z)\ge 4$ and $d_H(x,z),d_H(x,y),d_H(y,z)\ge 4$. If $\wt(y)=4$, then no such vector $z$ exists, so that $\wt(y),\wt(z)\ge 5$, which contradicts $d_H(y,z)\ge 4$.   
\end{proof}

\begin{lemma}
  \label{lemma_n_2_7_4_3}
  We have $m_2^{2}(6,3)=11$ and $m_2^{(2)}(7,3)=17$.
\end{lemma}
\begin{proof}
  An example showing $m_2^{(2)}(6,3)\ge 11$ is given by 
  $$
    \begin{pmatrix}
      1000000&11&10\\
      0100000&11&01\\
      0010000&10&10\\
      0001000&10&01\\
      0000100&01&10\\
      0000010&01&01\\
      0000001&00&11\\
    \end{pmatrix}.
  $$
  Let $\cM$ be a multiset of $n\ge 12$ points in $\PG(6,2)$ such that each plane contains at most three points. Since $m_2^{(2)}(5,3)=8$ we assume w.l.o.g.\ that $\cM$ contains the seven points spanned by the seven unit vectors. Clearly, the maximum point multiplicity is one and every additional point is spanned by a vector $x$ with Hamming weight at least $4$. Via a small ILP computation we excluded $n\ge 12$. 
  An example showing $m_2^{(2)}(7,3)\ge 17$ is given by 
  $$
    \begin{pmatrix}
      10000000&1110&01111\\
      01000000&1101&11101\\
      00100000&1010&10011\\
      00010000&1001&01110\\
      00001000&0110&10101\\
      00000100&0101&01011\\
      00000010&0011&00111\\
      00000001&0000&11111
    \end{pmatrix}.
  $$
  Again we can prescribe the eight points spanned by the unit vectors, so that any further point is spanned by a vector $x\in\F_2^8$ with Hamming weight at least $4$. If there is no point spanned by a vector with Hamming weight $4$, than a small ILP computations shows that the maximum cardinality is $16$. So we can additionally prescribe an arbitrary point spanned by a vector with Hamming weight $4$. Another ILP computation then shows that the maximum cardinality is $17$. 
\end{proof}

\begin{table}[htp]
  \begin{center}
    \begin{tabular}{rrcc}
      \hline 
      $w$ & $m_2^{(2)}(6,w)$ & construction & upper bound \\ 
      \hline
      3 &  11 & Lemma~\ref{lemma_n_2_7_4_3}  & Lemma~\ref{lemma_n_2_7_4_3} \\ 
      4 & 64 & $[6]-[5]$ & Griesmer upper bound \\ 
      5 & 65 & plus point & Griesmer upper bound \\
      6 & 96 & $[6]-[4]$ & Griesmer upper bound \\
      7 & 127 & $[6]$ & Griesmer upper bound \\
      8 & 128 & plus point & Griesmer upper bound \\
      9 & 129 & plus point & Griesmer upper bound \\
      10 & 160 & $2[6]-[5]-[4]$ & Griesmer upper bound \\
      \hline
    \end{tabular}  
    \caption{Exact values for $m_2^{(2)}(6,w)$.}
    \label{table_n_2_7_4_s}
  \end{center}
\end{table}

\begin{lemma}
  \label{lemma_n_2_6_2_11}
  We have $m_2^{(3)}(5,11)=38$.
\end{lemma}
\begin{proof}
  The lower bound is given by the unique $[38,6,18]_2$ code \cite{bouyukliev2001optimal}, which is a Griesmer code. A generator matrix is e.g.\ given by
  $$
    \left(\begin{smallmatrix}
    00000100000000000000001111111111111111\\
    00001000000001111111110000000111111111\\
    00010000011110000111110001111000001111\\
    00100001101110011000110110011001110011\\
    01000010110111101011011010101010010101\\
    10000011111010110101011101001100110110
    \end{smallmatrix}\right).
  $$
  Let $\cM$ be a multiset of points in $\PG(5,2)$ with at most $11$ points in each solid. 
  Considering projections through a point and $m_2^{2)}(4,9)=33$ implies $\cM(P)\le 1$ for every point $P$. 
  Since $m_2^{(3)}(4,11)=21$ we have 
  $\cM(H)\le 21$ for every hyperplane $H$. Since $m_2^{(4)}(5,20)=38$ we can assume the existence of a hyperplane $H^\star$ with $\cM(H^\star)=21$. There are exactly two $[21,5,10]_2$ codes, see e.g.\ \cite[Theorem 6]{aggarwal2012characterisation}. Prescribing these two possibilities for $H^\star$ a small ILP computation quickly shows $\#\cM\le 38$.    
\end{proof}
We remark that the Griesmer upper bound gives $m_2^{(3)}(5,11)\le 41$, while the coding upper bound yields $m_2^{(3)}(5,11)\le 39$ via 
$m_2^{(4)}(5,21)=39$. Note that the complement of a hypothetical set of $39$ points in $\PG(5,2)$ points with at most $11$ points per solid is a multiset of points of cardinality $24$ that blocks every solid at least four times. The union of a solid and a plane (plus two arbitrary points) gives such an example as a multiset but not as a set of points.

\begin{proposition}
  If $w\ge 12$ or $w\in\{6,8,9,10\}$, then $m_2^{(3)}(5,w)$ is given by the Griesmer upper bound. Moreover, we have $m_2^{3)}(5,4)=7$, $m_2^{(3)}(5,5)=11$,
  $m_2^{(3)}(5,7)=19$, and $m_2^{(3)}(5,11)=38$.
\end{proposition}
\begin{proof}
  Using Lemma~\ref{lemma_projective_base} and Lemma~\ref{lemma_n_2_6_2_11} we state that $m_2^{(3)}(5,4)=7$ and $m_2^{(3)}(5,11)=38$, respectively.
  We consider the following constructions
  \begin{itemize}
     \item types $t[5]$, $t[5]+[0]$, $t[5]+2[0]$, $t[5]+3[0]$, $t[5]-[4]$, $t[5]-[4]+[0]$, $t[5]-[4]+2[0]$, $t[5]-[3]$, $t[5]-[3]+[0]$, and $t[5]-[2]$ for $t\ge 1$;
     \item types $t[5]-[4]-[3]$, $t[5]-[4]-[2]$, and $t[5]-[3]-[2]$ for $t\ge 2$; and
     \item type $t[5]-[4]-[3]-[2]$ for $t\ge 3$.
  \end{itemize}
  So, it remains to provide constructions for $w\in\{5,\dots,7,19\}$. For $w=7$ we add an arbitrary point to an example for $w=6$. For $w\in\{5,6,19\}$ we provide explicit examples:
  $$
    \left(\begin{smallmatrix}
    00000100111\\
    00001011111\\
    00010001011\\
    00100011100\\
    01000011001\\
    10000011010
    \end{smallmatrix}\right),
    \left(\begin{smallmatrix}
    000000100001111111\\
    000011000110001111\\
    000100001110110011\\
    001000011010011101\\
    010000011101000111\\
    100000010111101001
    \end{smallmatrix}\right)
    ,\text{ and}
  $$
  $$
    \left(\begin{smallmatrix}
    0001000000000000111111111111111000000000001111111111000000000001111111111\\
    0010000001111111000001111111111000000111110000001111000000111110000001111\\
    0000001110000111001110000011111010011001110001110011010011001110001110011\\
    0000010110111000110010001100111101100001110001111100101100001110001111100\\
    0100111011001001010110010101011000101010110110010101000101010110110010101\\
    1000110011010011100100100110001001011100011010111001001011100011010111001
    \end{smallmatrix}\right).
  $$
  The Griesmer upper bound is not attained for $w\in \{4,5,7,11\}$ while it is for all other cases $w>3$ For $w\in\{5,7\}$ the coding upper bound is attained.
\end{proof}
The stored generator matrices of $[m_2^{(3)}(5,w),6]_2$ codes in the database of \emph{best known linear codes} (BKLC) in \texttt{Magma} give optimal examples for $w\in\{5,6,11,19\}$. 
Note that for $w=6$ we can take any $[18,6,8]_2$ Griesmer code and for $w=19$ we can take any $[73,6,36]_2$ Griesmer code.

\begin{table}[htp]
  \begin{center}
    \begin{tabular}{rrcc}
      \hline 
      $w$ & $m_2^{(3)}(5,w)$ & construction & upper bound \\ 
      \hline
      4 &  7 & projective base & Lemma~\ref{lemma_projective_base} \\ 
      5 & 11 & BKLC/ILP & Coding upper bound \\ 
      6 & 18 & BKLC/ILP & Griesmer upper bound \\
      7 & 19 & plus point & Coding upper bound \\
      8 & 32 & $[5]-[4]$ & Griesmer upper bound \\
      9 & 33 & plus point & Griesmer upper bound \\
      10 & 34 & plus point & Griesmer upper bound \\
      11 & 38 & BKLC/ILP & Lemma~\ref{lemma_n_2_6_2_11} \\
      12  & 48 & $[5]-[3]$ & Griesmer upper bound \\
      13  & 49 & plus point & Griesmer upper bound \\
      14  & 56 & $[5]-[2]$& Griesmer upper bound \\
      15 & 63 & $[5]$ & Griesmer upper bound \\
      16 & 64 & plus point & Griesmer upper bound \\
      17 & 65 & plus point & Griesmer upper bound \\
      18 & 66 & plus point  & Griesmer upper bound \\
      19 & 73 & BKLC/ILP & Griesmer upper bound \\
      20 & 80 & $2[5]-[4]-[3]$ & Griesmer upper bound \\
      21 & 81 & plus point & Griesmer upper bound \\
      22 & 88 & $2[5]-[4]-[2]$ & Griesmer upper bound \\
      23 & 95 & sum construction & Griesmer upper bound \\
      24 & 96 & plus point & Griesmer upper bound \\
      25 & 97 & plus point & Griesmer upper bound \\
      26 & 104 & $2[5]-[3]-[2]$ & Griesmer upper bound \\
      27 & 111 & sum construction & Griesmer upper bound \\
      28 & 112 & plus point & Griesmer upper bound \\
      29 & 119 & sum construction & Griesmer upper bound \\
      30 & 126 & sum construction & Griesmer upper bound \\
      31 & 127 & plus point & Griesmer upper bound \\
      32 & 128 & plus point & Griesmer upper bound \\
      33 & 129 & plus point & Griesmer upper bound \\
      34 & 136 & $3[5]-[4]-[3]-[2]$ & Griesmer upper bound \\
      \hline
    \end{tabular}  
    \caption{Exact values for $m_2^{(3)}(5,w)$.}
    \label{table_n_2_6_2_s}
  \end{center}
\end{table}

\begin{lemma}
  \label{lemma_n_2_7_s_9_and_10}
  We have $m_2^{(4)}(6,9)=27$ and  $m_2^{(4)}(6,10)=28$.
\end{lemma}
\begin{proof}
  An example showing $m_2^{(4)}(6,9)\ge 27$ is given by $$
    \left(\begin{smallmatrix}
    000000100000000011111111111\\
    000001000001111100000111111\\
    000010001110111100011000011\\
    000100010111000101111000101\\
    001000011001001110101011100\\
    010000001011011010110101010\\
    100000010100110111010110001
    \end{smallmatrix}\right),
  $$
  so that adding an arbitrary point gives $m_2^{(4)}(6,10)\ge 28$. For $w=9$ the coding upper bound is attained. Next we consider a multiset $\cM$ of points in $\PG(6,2)$ such that every subspace of codimension two contains at most ten points. Starting from $m_2^{(5)}(6,10)=18$ we have used \texttt{LinCode} \cite{bouyukliev2021computer} to enumerate the two non-isomorphic $[18,6,8]_2$ codes. 
Prescribing the two possible configurations and an arbitrary point making the span $7$-dimensional, we have used an ILP computation to conclude $\#\cM\le 28$. In the following we assume $\cM(H)\le 17$ for each hyperplane. Since $m_2^{(3)}(5,7)=19$ we also assume $\cM(P)\le 2$ for every point $P$. If there exists a solid $S$ with $\cM(S)\ge 7$, then we have $\#\cM \le 7\cdot 3+7=28$, so that we assume $\cM(S)\le 6$ for every solid $S$. We have used \texttt{LinCode} \cite{bouyukliev2021computer} to enumerate the four non-isomorphic $[10,5,4]_2$ codes. 
  Prescribing the four possible configurations and two points that make the span $7$-dimensional, we have used ILP computations to conclude $\#\cM\le 28$.
\end{proof}

\begin{lemma}
  \label{lemma_n_2_7_s_20}
  We have $m_2^{(4)}(6,20)=71$.
\end{lemma}
\begin{proof}
  An example showing $m_2^{(4)}(6,20)\ge 71$ is given by
  $$ 
    \left(\begin{smallmatrix}
    00000010000000000000000000000000000000111111111111111111111111111111111\\
    00000100000000000000011111111111111111000000000000000011111111111111111\\
    00001000000011111111100000000011111111000000001111111100000000011111111\\
    00010000111100001111100000111100001111000011110000111100000111100001111\\
    00100001001101110001100011001100110011001100110011001100111001100110011\\
    01000001010110110010101101010101010101110101010101010101011010101010101\\
    10000001111011010100110110100110010110011010011001011010001011001101001
    \end{smallmatrix}\right).
  $$
  Next we consider a multiset $\cM$ of points in $\PG(6,2)$ such that every subspace of codimension two contains at most $20$ points. Starting from $m_2^{(4)}(5,20)=38$ we have used \texttt{LinCode} \cite{bouyukliev2021computer} to construct the unique  $[38,6,18]_2$ code \cite{bouyukliev2001optimal}.
  Prescribing the corresponding unique configuration and a further point that makes the span $7$-dimensional, we have used an ILP computation to conclude $\#\cM\le 71$.
  Observing $m_2^{(5)}(6,37)=71$ finishes the proof.
\end{proof}

\begin{lemma}
  \label{lemma_n_2_7_s_22_and_23}
  We have $m_2^{(4)}(6,22)=82$ and  $m_2^{(4)}(6,23)=83$.
\end{lemma}
\begin{proof}
  An example showing $m_2^{(4)}(6,22)\ge82$ is given by each of the $11$ $[82,7,40]_2$ Griesmer codes \cite{bouyukliev2001optimal}. One generator matrix is e.g.\ given by 
  $$ 
    \left(\begin{smallmatrix}
    0000001000000000000000000000000000000000000111111111111111111111111111111111111111\\
    0000010000000000000000001111111111111111111000000000000000000011111111111111111111\\
    0000100000011111111111110000000111111111111000000011111111111100000000111111111100\\
    0001000011100000011111110001111000001111111000111100000001111100001111000001111100\\
    0010000101100011100011110110011000110000111011001100011110011100110011001110001100\\
    0100000110101100101100111010101001010011001101010101100110101101010101010110010100\\
    1000000111010101010101011101001010010101010110100110101010110110010110011010100100
    \end{smallmatrix}\right),
  $$
  so that adding an arbitrary point gives $m_2^{(4)}(6,23)\ge 83$. For $s=22$ the Griesmer upper bound is attained. Next we consider a multiset $\cM$ of points in $\PG(6,2)$ such that every subspace of codimension two contains at most $23$ points. Starting from $m_2^{(4)}(5,23)=34$ we have used \texttt{LinCode} \cite{bouyukliev2021computer} to construct the unique $[45,6,22]_2$ code \cite{bouyukliev2001optimal}.
  Prescribing the corresponding configuration and an arbitrary point making the span $7$-dimensional, we have used an ILP computation to conclude $\#\cM\le 83$. In the following we assume $\cM(H)\le 44$ for each hyperplane. Since $m_2^{(3)}(5,20)=80$ we also assume $\cM(P)\le 1$ for every point $P$. We have used \texttt{LinCode} \cite{bouyukliev2021computer} to enumerate the unique $[44,6,21]_2$ code with maximum column multiplicity one. 
  Prescribing the corresponding configuration and an arbitrary point making the span $7$-dimensional, we have used an ILP computation to conclude $\#\cM\le 83$. In the following we assume $\cM(H)\le 43$ for each hyperplane. 
  If there exists a solid $S$ with $\cM(S)\ge 13$, then we have $\#\cM \le 7\cdot 10+13=83$, so that we assume $\cM(S)\le 12$ for every solid $S$. We have used \texttt{LinCode} \cite{bouyukliev2021computer} to construct the unique $[23,11,5]_2$ code \cite{simonis200023}. 
  Prescribing the corresponding configuration and two points that make the span $7$-dimensional, we have used an ILP computation to conclude $\#\cM\le 83$.
\end{proof}

\begin{proposition}
  If $w\ge 24$ or $w\in\{7,11,14,16,\dots,19,22\}$, then $m_2^{(4)}(6,s)$ is given by the Griesmer upper bound. Moreover, we have $m_2^{(4)}(6,5)=8$, $m_2^{(4)}(6,6)=12$, 
$m_2^{(4)}(6,8)=20$, $m_2^{(4)}(6,9)=27$, $m_2^{(4)}(6,10)=28$, $m_2^{(4)}(6,12)=36$, $_2^{(4)}(6,13)=43$, $m_2^{(4)}(6,15)=51$, $m_2^{(4)}(6,20)=71$, $m_2^{(4)}(6,21)=75$, and 
$m_2^{(4)}(6,23)=83$.
\end{proposition}
\begin{proof}
  For $m_2^{(4)}(6,9)=27$ and $m_2^{(4)}(6,10)=28$ we refer to Lemma~\ref{lemma_n_2_7_s_9_and_10}. For $m_2^{(4)}(6,20)=71$ we refer to Lemma~\ref{lemma_n_2_7_s_20}. 
  For $m_2^{(4)}(6,22)=82$ and $m_2^{(4)}(6,23)=83$ we refer to Lemma~\ref{lemma_n_2_7_s_22_and_23}. 
  We consider the following constructions
  \begin{itemize}
     \item types $t[6]$, $t[6]-[5]$, $t[6]-[4]$, $t[6]-[3]$, and $t[6]-[2]$
           for $t\ge 1$;
     \item types $t[6]-[5]-[4]$, $t[6]-[5]-[3]$, $t[6]-[5]-[2]$,
          $t[6]-[4]-[3]$, $t[6]-[4]-[2]$, and $t[6]-[3]-[2]$ for $t\ge 2$;
     \item types $t[6]-[5]-[4]-[3]$, $t[6]-[5]-[4]-[2]$, $t[6]-[5]-[3]-[2]$,
           and $t[6]-[4]-[3]-[2]$ for $t\ge 3$; and
     \item type $t[6]-[5]-[4]-[3]-[2]$ for $t\ge 4$,
  \end{itemize}
  as well as adding up to four additional points to those constructions. By selecting the removed subspaces more carefully than by a chain, we can also have constructions for $[6]-[3]-[2]$, $2[6]-[5]-[3]-[2]$, $2[6]-[4]-[3]-[2]$, and $3[6]-[5]-[4]-[3]-[2]$, i.e.\ for $w\in\{27,43,51,67\}$, which meet the Griesmer upper bound.  The case $w=5$ is treated in Lemma~\ref{lemma_projective_base}. For $w\in\{8,12,15,38\}$ the best known construction can be obtained by adding an arbitrary point to a construction for $w-1$. For $w\in\{6,7,11,13,14,21\}$ 
  we give explicit examples found by ILP searches.
  $$
    \left(\begin{smallmatrix}
    000000100111\\
    000001001011\\
    000010010011\\
    000100011100\\
    001000011111\\
    010000001101\\
    100000011001
    \end{smallmatrix}\right),
    \left(\begin{smallmatrix}
    0000001000001111111\\
    0000010001110001111\\
    0000100110010010111\\
    0001000011100110011\\
    0010000010111100101\\
    0100000101101101001\\
    1000000011011011001
    \end{smallmatrix}\right),\\
    \left(\begin{smallmatrix}
    00000010000000000000111111111111111\\
    00000100000111111111000001111111111\\
    00001000011000011111001110000011111\\
    00010001111001100111010110001100001\\
    00100000101110100011011010010101110\\
    01000001011011111001100100101000110\\
    10000000110011101010100111000111010
    \end{smallmatrix}\right),\\
  $$
  $$
    \left(\begin{smallmatrix}
    0000001000000000000011111111111111111111111\\
    0000010000011111111100000000011111111111111\\
    0000100001100001111100001111100000001111111\\
    0001000110101110011100110111100011110001111\\
    0010000011010110001111001001100100110010111\\
    0100000111011000100101010010101001010111011\\
    1000000100101011100110111000110100010011101
    \end{smallmatrix}\right),
  $$
  $$
    \left(\begin{smallmatrix}
    00000010000000000000000001111111111111111111111100\\
    00000100000000111111111110000000000011111111111100\\
    00001001111111000000011110000000111100001111111100\\
    00010000000111000011101110011111000100010001111111\\
    00100000011001001100110110100011011101110010001111\\
    01000000101010011111000011101101001010110100010111\\
    10000001010010100101011010110100101111011000100111
    \end{smallmatrix}\right),
  $$
  $$ 
    \left(\begin{smallmatrix}
    000000100000000000000000000000000000001111111111111111111111111111111111111\\
    000001000000000000000111111111111111110000000000000000000111111111111111111\\
    000010000000001111111000000001111111110000000001111111111000000000111111111\\
    000100000111110001111000111110000111110000111110000111111000001111000011111\\
    001000001001110110011011000110011000110011000110011000111001110011001100111\\
    010000011110011010101101011010101011011101011010101001011010110101010101001\\
    100000011010101101001110101011001001100110101011001011101100010110011010011
   \end{smallmatrix}\right),
  $$
  For $w\in\{36,37,39\}$ we can take $[138,7,68]_2$, $[145,7,72]_2$, and $[153,7,76]_2$ Griesmer codes \cite{van1981smallest}. 
  For $w\in\{6,8,12,13,15,21\}$ the coding upper bound is attained.
  All other upper bounds are given by the Griesmer upper bound.
\end{proof}

\begin{table}[htp]
  \begin{center}
    \begin{tabular}{rrcc}
      \hline 
      $w$ & $m_2^{(4)}(6,w)$ & construction & upper bound \\ 
      \hline
      5 &  8 & projective base & Lemma~\ref{lemma_projective_base} \\ 
      6 & 12 & BKLC/ILP & Coding upper bound \\ 
      7 & 19 & BKLC/ILP & Griesmer upper bound \\
      8 & 20 & plus point & Coding upper bound \\
      9 & 27 & BKLC/ILP & Coding upper bound \\
      10 & 28 & plus point & Lemma~\ref{lemma_n_2_7_s_9_and_10} \\
      11 & 35 & BKLC/ILP & Griesmer upper bound \\
      12 & 36 & plus point & Coding upper bound \\
      13 & 43 & BKLC/ILP & Coding upper bound \\
      14 & 50 & BKLC/ILP & Griesmer upper bound \\
      15 & 51 & plus point & Coding upper bound \\
      16 & 64 & $[6]-[5]$ & Griesmer upper bound \\
      17 & 65 & plus point& Griesmer upper bound \\
      18 & 66 & plus point & Griesmer upper bound \\
      19 & 67 & plus point & Griesmer upper bound \\
      20 & 71 & BKLC/ILP & Lemma~\ref{lemma_n_2_7_s_20} \\
      21 & 75 & BKLC/ILP & Coding upper bound \\
      22 & 82 & BKLC/ILP & Griesmer upper bound \\
      23 & 83 & plus point & Lemma~\ref{lemma_n_2_7_s_22_and_23}\\
      24 & 96 & $[6]-[4]$ & Griesmer upper bound \\
      25 & 97 & plus point & Griesmer upper bound \\
      26 & 98 & plus point & Griesmer upper bound \\
      27 & 105 & $[6]-[3]-[2]$ & Griesmer upper bound \\
      28 & 112 & $[6]-[3]$ & Griesmer upper bound \\
      29 & 113 & plus point & Griesmer upper bound \\
      30 & 120 & $[6]-[2]$ & Griesmer upper bound \\
      31 & 127 & $[6]$ & Griesmer upper bound \\
      32 & 128 & plus point & Griesmer upper bound \\
      33 & 129 & plus point & Griesmer upper bound \\
      34 & 130 & plus point & Griesmer upper bound \\
      35 & 131 & plus point & Griesmer upper bound \\
      36 & 138 & BKLC/ILP & Griesmer upper bound \\
      37 & 145 & BKLC/ILP & Griesmer upper bound \\
      38 & 146 & plus point & Griesmer upper bound \\
      39 & 153 & BKLC/ILP & Griesmer upper bound \\
      40 & 160 & $2[6]-[5]-[4]$& Griesmer upper bound \\
      41 & 161 & plus point & Griesmer upper bound \\
      42 & 162 & plus point & Griesmer upper bound \\
      43 & 169 & $2[6]-[5]-[3]-[2]$ & Griesmer upper bound \\
      44 & 176 & $2[6]-[5]-[3]$ & Griesmer upper bound \\
      45 & 177 & plus point & Griesmer upper bound \\
      \hline
    \end{tabular}  
    \caption{Exact values for $m_2^{(4)}(6,w)$ -- part 1.}
    \label{table_n_2_7_2_s_part1}
  \end{center}
\end{table}

\begin{table}[htp]
  \begin{center}
    \begin{tabular}{rrcc}
      \hline 
      $w$ & $m_2^{(4)}(6,w)$ & construction & upper bound \\ 
      \hline
      46 & 184 & $2[6]-[5]-[2]$ & Griesmer upper bound \\
      47 & 191 & sum construction & Griesmer upper bound \\
      48 & 192 & plus point & Griesmer upper bound \\
      49 & 193 & plus point & Griesmer upper bound \\
      50 & 194 & plus point & Griesmer upper bound \\
      51 & 201 & $2[6]-[4]-[3]-[2]$ & Griesmer upper bound \\
      52 & 208 & $2[6]-[4]-[3]$ & Griesmer upper bound \\
      53 & 209 & plus point & Griesmer upper bound \\
      54 & 216 & $2[6]-[4]-[2]$ & Griesmer upper bound \\
      55 & 223 & sum construction & Griesmer upper bound \\
      56 & 224 & plus point & Griesmer upper bound \\
      57 & 225 & plus point & Griesmer upper bound \\
      58 & 232 & sum construction & Griesmer upper bound \\
      59 & 239 & sum construction & Griesmer upper bound \\
      60 & 240 & plus point & Griesmer upper bound \\
      61 & 247 & sum construction & Griesmer upper bound \\
      62 & 254 & sum construction & Griesmer upper bound \\
      63 & 255 & plus point & Griesmer upper bound \\
      64 & 256 & plus point & Griesmer upper bound \\
      65 & 257 & plus point & Griesmer upper bound \\
      66 & 258 & plus point & Griesmer upper bound \\
      67 & 265 & $3[6]-[5]-[4]-[3]-[2]$ & Griesmer upper bound \\
      68 & 272 & $3[6]-[5]-[4]-[3]$ & Griesmer upper bound \\
      69 & 273 & plus point & Griesmer upper bound \\
      70 & 280 & $3[6]-[5]-[4]-[2]$ & \\
      \hline
    \end{tabular}  
    \caption{Exact values for $m_2^{(4)}(6,w)$ -- part 2.}
    \label{table_n_2_7_2_s_part2}
  \end{center}
\end{table}

The stored generator matrices of $[m_2^{(4)}(6,w),7]_2$ codes in the database of \emph{best known linear codes} (BKLC) in \texttt{Magma} give optimal examples for \[w\in\{6,7,9,11,13,14,20,21,22,36,37,39\}.\]

\begin{lemma}
  \label{lemma_n_2_7_3_s_4_and_5}
  We have $m_2^{(3)}(6,4)=9$ and $m_2^{(3)}(6,5)=19$.
\end{lemma}
\begin{proof}
  Examples showing $m_2^{(3)}(6,4)\ge9$ and $m_2^{(3)}(6,5)\ge19$
  are given by 
  $$
    \begin{pmatrix}
    0&0&0&0&0&1&1&1&1\\ 
    0&0&0&1&1&0&0&1&1\\ 
    0&0&0&0&0&0&1&0&1\\ 
    0&0&0&0&1&1&1&0&1\\ 
    1&1&1&0&0&0&1&1&0\\ 
    0&1&1&0&0&0&0&0&0\\ 
    1&1&0&0&0&0&0&0&0\\
    \end{pmatrix}
  $$
  and
  $$
    \begin{pmatrix}
    1000000000111101101\\
    0100000101010111001\\
    0010000111100010011\\
    0001000110111000110\\
    0000100011011100011\\
    0000010100100111110\\
    0000001010010011111\\
    \end{pmatrix},
  $$
  respectively. For upper bounds are obtained by ILP computations prescribing the seven points generated by the unit vectors in $\F_2^7$.
\end{proof}

\begin{lemma}
  \label{lemma_n_2_7_3_s_6}
  We have $m_2^{(3)}(6,6)=28$.
\end{lemma}
\begin{proof}
  An example showing $m_2^{(3)}(6,6)\ge 28$ is given by $$
    \left(\begin{smallmatrix}
    0 0 0 0 0 0 0 0 0 0 0 0 1 1 1 1 1 1 0 0 0 0 1 1 1 1 1 1\\ 
    0 0 0 0 0 0 0 0 0 1 1 1 0 0 0 1 1 1 0 1 1 1 0 0 0 1 1 1\\ 
    0 0 0 1 1 1 1 1 1 0 0 0 0 0 0 1 1 1 0 0 0 1 0 0 1 0 1 1\\ 
    1 1 1 0 0 0 1 1 1 0 0 0 0 0 0 0 0 0 0 0 1 1 0 1 1 1 0 1\\ 
    0 0 0 1 1 1 0 0 0 0 0 0 1 1 1 1 1 1 1 1 1 1 0 0 0 0 0 0\\ 
    0 1 1 0 1 1 0 1 1 0 1 1 0 1 1 0 1 1 0 0 0 0 0 0 0 0 0 0\\ 
    1 1 0 1 1 0 1 1 0 1 1 0 1 1 0 1 1 0 0 0 0 0 0 0 0 0 0 0
\end{smallmatrix}\right).
  $$
  Let $\cM$ be a multiset of $n$ points in $\PG(6,2)$ such that every solid contains at most six points. We have used \texttt{LinCode} \cite{bouyukliev2021computer} to enumerate the two non-isomorphic $[18,6,8]_2$ codes. 
  Prescribing the two possible configurations and an arbitrary point making the span $7$-dimensional, we have used an ILP computation to conclude $\#\cM\le 25$. In the following we assume $\cM(H)\le 17$ for each hyperplane. 
  We have used \texttt{LinCode} \cite{bouyukliev2021computer} to enumerate the three non-isomorphic $[17,6,7]_2$ codes.
  Prescribing the three possible configurations and an arbitrary point making the span $6$-dimensional, we have used an ILP computation to conclude $\#\cM\le 28$. In the following we assume $\cM(H)\le 16$ for each hyperplane. 
  We have used \texttt{LinCode} \cite{bouyukliev2021computer} to enumerate the four non-isomorphic $[10,5,4]_2$ codes. 
  Prescribing the four possible configurations and two points that make the span $6$-dimensional, we have used ILP computations to conclude $\#\cM\le 28$.
\end{proof}

\begin{lemma}
  \label{lemma_n_2_7_3_s_11}
  We have $m_2^{(3)}(6,11)=72$.
\end{lemma}
\begin{proof}
  An example showing $m_2^{(3)}(6,11)\ge 72$ is given by $$
    \left(\begin{smallmatrix}
    100000011011100011101001110000110011101011010001010111001001011111111100\\
010000010110010010011101001000101010011110111001111100101101110010000010\\
001000001111011001101110001100001101111010010100011101010010111001000010\\
000100000010001101000111101110010110111101001010100111001101111000100010\\
000010000000110111110111000111000011001111001101011010100010111100010010\\
000001000001001010001111110011101001110110100110001110110101011100001010\\
000000111110100111011010010001010111110100100010111010010011110000000110\\
\end{smallmatrix}\right).
  $$
  Let $\cM$ be a multiset of $n$ points in $\PG(6,2)$ such that every solid contains at most eleven points. We have $\cM(H)\le m_2^{(3)}(5,11)=38$ for every hyperplane $H$, so that $\#\cM\le m_2^{(5)}(6,38)\le 72$.
\end{proof}

\begin{proposition}
  If $w\ge 12$ or $w\in\{8,9,10\}$, then $m_2^{(3)}(6,w)$ is given by the Griesmer upper bound. Moreover, we have $m_2^{(4)}(6,4)=9$, $m_2^{(4)}(6,5)=19$, $m_2^{(4)}(6,6)=28$, 
 $m_2^{(4)}(6,7)=35$, and  $m_2^{(4)}(6,11)=72$.
\end{proposition}
\begin{proof}
  For $m_2^{(3)}(6,4)=9$ and $m_2^{(3)}(6,5)=19$ we refer to Lemma~\ref{lemma_n_2_7_3_s_4_and_5}. For $m_2^{(3)}(6,6)=28$ we refer to Lemma~\ref{lemma_n_2_7_3_s_6}. For $m_2^{(3)}(6,11)=72$ we refer to Lemma~\ref{lemma_n_2_7_3_s_11}. 
  We consider the following constructions
  \begin{itemize}
     \item types $t[6]$, $t[6]-[5]$, $t[6]-[4]$, and $t[6]-[3]$
           for $t\ge 1$;
     \item types $t[6]-[5]-[4]$, $t[6]-[5]-[3]$, $t[6]-[5]$,
          $t[6]-[4]-[3]$, and  $t[6]-[4]$ for $t\ge 2$,
  \end{itemize}
  as well as adding up to four additional points to those constructions. Examples showing $m_2^{(3)}(6,7)\ge 35$ 
  given by
  $$
    \left(\begin{smallmatrix}
    1 0 0 1 0 1 0 1 1 0 1 0 1 0 1 1 0 1 0 1 1 1 0 1 0 1 1 0 0 0 0 1 1 1 1\\
    0 1 0 1 0 0 0 1 0 1 1 0 1 1 0 0 0 1 1 1 0 0 0 0 1 0 1 0 1 1 1 0 1 0 0\\
    0 0 1 1 0 1 0 0 0 0 1 0 0 1 0 0 0 0 1 1 0 0 0 1 1 1 0 1 0 0 1 1 1 1 1\\
    0 0 0 0 1 1 0 0 1 1 1 0 0 0 1 0 0 0 0 1 0 0 1 0 1 1 1 0 1 1 0 1 1 0 1\\
    0 0 0 0 0 0 1 1 1 1 1 0 0 0 0 1 0 0 0 0 1 0 0 1 1 1 1 1 0 1 1 0 1 1 0\\
    0 0 0 0 0 0 0 0 0 0 0 1 1 1 1 1 0 0 0 0 0 1 1 1 1 1 1 1 1 0 0 0 1 1 1\\
    0 0 0 0 0 0 0 0 0 0 0 0 0 0 0 0 1 1 1 1 1 1 1 1 1 1 1 1 1 1 1 1 0 0 0   
    \end{smallmatrix}\right).
  $$
  For $m_2^{(3)}(6,19)\ge 145$ we can use a $[145,7,72]_2$ Griesmer code \cite{van1981smallest}.
  The upper bound $m_2^{(3)}(6,7)\le 35$ is given by the coding upper bound. All other upper bounds are obtained from the Griesmer upper bound.
\end{proof}  

The stored generator matrices of $[m_2^{(3)}(6,w),7]_2$ codes in the database of \emph{best known linear codes} (BKLC) in \texttt{Magma} give optimal examples for $w\in\{5,7,19\}$.

\begin{table}[htp]
  \begin{center}
    \begin{tabular}{rrcc}
      \hline 
      $w$ & $m_2^{(3)}(6,w)$ & construction & upper bound \\ 
      \hline
       4 & 9 & ILP & Lemma~\ref{lemma_n_2_7_3_s_4_and_5} \\
       5 & 19 & BKLC/ILP & Lemma~\ref{lemma_n_2_7_3_s_4_and_5} \\
       6 & 28 & ILP & Lemma~\ref{lemma_n_2_7_3_s_6} \\
       7 & 35 & BKLC/ILP & Coding upper bound \\
       8 & 64 & $[6]-[5]$ & Griesmer upper bound \\
       9 & 65 & plus point & Griesmer upper bound \\
      10 & 66 & plus point & Griesmer upper bound \\
      11 & 72 & Lemma~\ref{lemma_n_2_7_3_s_11} & Lemma~\ref{lemma_n_2_7_3_s_11} \\
      12 & 96 & $[6]-[4]$ & Griesmer upper bound \\
      13 & 97 & plus point & Griesmer upper bound \\
      14 & 112 & $[6]-[3]$ & Griesmer upper bound \\
      15 & 127 & $[6]$ & Griesmer upper bound \\
      16 & 128 & plus point & Griesmer upper bound \\
      17 & 129 & plus point & Griesmer upper bound \\
      18 & 130 & plus point & Griesmer upper bound \\
      19 & 145 & BKLC/ILP & Griesmer upper bound \\
      20 & 160 & $2[6]-[5]-[4]$ & Griesmer upper bound \\
      21 & 161 & plus point & Griesmer upper bound \\
      22 & 176 & $2[6]-[5]-[3]$ & Griesmer upper bound \\
      23 & 191 & sum construction & Griesmer upper bound \\
      24 & 192  & sum construction & Griesmer upper bound \\
      25 & 193 & sum construction & Griesmer upper bound \\
      26 & 208 & $2[6]-[4]-[3]$ & Griesmer upper bound \\
      \hline
    \end{tabular}  
    \caption{Exact values for $m_2^{(3)}(6,w)$.}
    \label{table_n_2_7_3_s}
  \end{center}
\end{table}

\subsection{Exact values for $m_3^{(r)}(k-1,w)$}
\label{subsec_ternary}

\begin{proposition}
  If $w\ge 3$ , then $m_3^{(1)}(3,w)$ is given by the Griesmer upper bound. Moreover, we have $m_3^{(1)}(3,2)=10$.
\end{proposition}
\begin{proof}
  The upper bound $m_3^{(1)}(3,2)\le 10$ follows from the coding upper bound and all other upper bounds follow from the Griesmer upper bound. The existence of an ovoid in $\PG(3,3)$ yields $m_3^{(1)}(3,2)\ge 10$. A $[27,4,18]_3$ Griesmer code yields $m_3^{(1)}(3,3)\ge 27$. Note that Griesmer $[n,4,d]_3$ codes exist for all $d\ge 16$. 
\end{proof}

We remark that we have $m_q^{(1)}(3,2)=q^2+1$ for all $q>2$ \cite{bose1947mathematical,qvist1952some}. Moreover, we have 
$m_4^{(1)}(3,3)=31$ and $m_5^{(1)}(3,3)=44$ \cite{edel2010multiple}. 

\begin{lemma}
  \label{lemma_3_5_2_4}
  We have $m_3^{(2)}(4,4)=20$.
\end{lemma}
\begin{proof}
  An example showing $m_3^{(2)}(4,4)\ge 20$ is given by
  $$ 
    \begin{pmatrix}
    10000220001102111221\\
    01000002221001121111\\
    00100212202022200211\\
    00010202111222011010\\
    00001111021120010111  
    \end{pmatrix}.
  $$
  After prescribing the unique $[10,4;6]_3$ code a small ILP computation verifies 
  $m_3^{(2)}(4,4)\le 20$. 
\end{proof}

\begin{lemma}
  \label{lemma_3_5_2_6}
  We have $m_3^{(2)}(4,6)=38$.
\end{lemma}
\begin{proof}
  An example showing $m_3^{(2)}(4,6)\ge 38$ is given by
  $$ 
    \left(\begin{smallmatrix}
    1 0 0 0 0 0 1 2 0 1 0 1 1 1 2 2 1 2 0 2 1 1 0 0 0 2 0 0 1 2 1 2 2 0 1 2 1 1\\
    0 1 0 0 0 2 1 2 1 2 2 1 2 2 0 1 0 2 2 2 2 1 2 2 1 0 1 0 1 0 1 1 0 1 0 1 1 1\\
    0 0 1 0 0 2 0 1 2 0 0 1 2 1 1 2 0 1 0 2 2 0 1 1 2 1 1 2 0 2 0 0 1 0 1 2 2 1\\
    0 0 0 1 0 1 2 0 1 2 0 1 0 1 2 2 1 2 2 1 2 0 1 0 2 2 0 2 0 1 2 2 1 1 0 0 0 2\\
    0 0 0 0 1 2 2 2 2 2 2 2 2 2 2 2 2 2 2 2 0 1 0 1 1 1 2 2 2 1 0 2 2 2 1 2 0 2]\end{smallmatrix}\right).
  $$
  After prescribing the three non-equivalent  $[15,4;9]_3$ codes  small ILP computations verify $m_3^{(2)}(4,6)\le 38$. 
\end{proof}

\begin{lemma}
  \label{lemma_3_5_2_11}
  We have $m_3^{(2)}(4,11)=91$.
\end{lemma}
\begin{proof}
  An example showing $m_3^{(2)}(4,6)\ge 38$ is given by the $[91,5,60]_3$ code in the database of \emph{best known linear codes} (BKLC) in \texttt{Magma}.
  After prescribing the unique non-equivalent  $[32,4;21]_3$ code an ILP computation verifies $m_3^{(2)}(4,11)\le 91$. 
\end{proof}

\begin{proposition}
  If $w\ge 18$ or $s\in\{7,9,10,12,13,14,15,16\}$, then $m_3^{(2)}(4,w)$ is given by the Griesmer upper bound. Moreover, we have $m_3^{(2)}(4,3)=11$, $m_3^{(2)}(4,4)=20$, $m_3^{(2)}(4,5)=29$, $m_3^{(2)}(4,6)=38$, $m_3^{(2)}(4,8)=56$, $m_3^{(2)}(4,11)=91$, and 
 $m_3^{(2)}(4,17)\in\{143,\dots,146\}$.
\end{proposition}
\begin{proof}
  For $m_3^{(2)}(4,4)=20$ we refer to Lemma~\ref{lemma_3_5_2_4} and for 
  $m_3^{(2)}(4,6)=38$ we refer to Lemma~\ref{lemma_3_5_2_6}. For $m_3^{(2)}(4,11)=91$ we refer to Lemma~\ref{lemma_3_5_2_11}.
  For $w\in\{7,9,12,13,16\}$ the existence of $[55,5,36]_3$, $[81,5,54]_3$, $[108,5,72]_3$, $[121,5,81]_3$, and $[136,5,90]_3$ Griesmer codes yields the lower bounds.
  For $w\in \{10,14,15\}$ the lower bound is attained by adding arbitrary points.  Also 
  $m_3^{(2)}(4,8)\ge 56$ is given by adding a point. 
  Since Griesmer $[n,5,d]_3$ codes do exist for all $d\ge 100$, $m_3^{(2)}(4,w)$ is given by the Griesmer upper bound for all $w\ge 18$. 
  For $w\in\{3,4,5,6,11\}$ the stored generator matrices of $[n,5]_3$ codes in the database of \emph{best known linear codes} (BKLC) in \texttt{Magma}
  yield $m_3^{(2)}(4,3)\ge 11$,  $m_3^{(2)}(4,4)\ge 20$, $m_3^{(2)}(4,5)\ge 29$, 
  $m_3^{(2)}(4,6)\ge 38$, and $m_3^{(2)}(4,11)\ge 91$, respectively.

  The coding upper bound gives $m_3^{(2)}(4,3)\le 11$, 
  $m_3^{(2)}(4,5)\le 29$, 
  $m_3^{(2)}(4,8)\le 56$, 
  and $m_3^{(2)}(4,17)\le 146$. All other upper bounds are given by the Griesmer upper bound.
\end{proof}

\begin{proposition}
  If $w\ge 3$, then $m_3^{(1)}(4,w)$ is given by the Griesmer upper bound. Moreover, we have $m_3^{(1)}(4,2)=20$.
\end{proposition}
\begin{proof}
  For $m_3^{(1)}(4,2)=20$ we refer e.g.\ to \cite{pellegrino1971, hill1983pellegrino}.\footnote{Up to projective equivalence there are exactly nine different examples.}
  The $[81,5,54]_3$ and the $[121,5,81]_3$ Griesmer codes give examples for $m_3^{(1)}(4,3)\ge 81$ and $m_3^{(1)}(4,4)\ge 121$, respectively. $m_3^{(1)}(4,5)\ge 122$ is obtained by adding an arbitrary point. The other lower bounds follow from the fact that Griesmer $[n,5,d]_3$ codes exist for all $d\ge 100$. The upper bounds are given by the Griesmer upper bound except for $w=2$.
 \end{proof}

We have $m_4^{(1)}(4,2)=41$ \cite{tallini1964calotte,edel199941}{\footnote{Up to symmetry there exist two $41$-caps in $\PG(4,4)$.} and
$m_3^{(1)}(5,2)=56$ \cite{hill1973largest}\footnote{Up to symmetry there exists a unique $56$-cap in $\PG(5,3)$.}.

\section*{Acknowledgments}
%
The second author was  supported  by  the
Bulgarian  National  Science  Research  Fund
under  Grant  KP-06-N72/6-2023. The third author was supported by the
Research Fund of Sofia University under Contract 80-10-14/21.05.2025.


\end{document}